\documentclass[11pt]{amsart}

\usepackage[utf8]{inputenc}
\usepackage{amsmath,amssymb,amsthm, array, bm}
\usepackage{hyperref}
\usepackage{xcolor}

\newtheorem{theorem}{Theorem}[section]
\newtheorem{lemma}[theorem]{Lemma}
\newtheorem{corollary}[theorem]{Corollary}
\newtheorem{remark}[theorem]{Remark}

\title{Cusp volumes and rational points on spheres}
\author{Konrad Zyhalko}
\address{Technical University of Munich, Germany}
\email{konrad.zyhalko@tum.de}

\begin{document}
\maketitle

\begin{abstract}
In recent work on the Fourier coefficients of light-cone Eisenstein series arising in the problem of counting rational points on spheres, Kelmer and Yu conjectured that the leading coefficient in the asymptotic formula for the counting function is rational up to a volume factor. We prove this conjecture by showing that every cusp volume appearing in the corresponding light-cone Eisenstein series is rational.
\end{abstract}

\section{Introduction}

Counting rational points with bounded denominators on $n$-dimensional spheres is a classical topic at the intersection of analytic and algebraic number theory. The investigation of this counting problem may be conducted using many different approaches, including the circle method \cite{94587f09-97a1-3a36-8ad9-ae4374902ee0, grosswald1985representations}, harmonic analysis on groups \cite{10.1215/S0012-7094-93-07107-4}, analysis of modular forms \cite{burrin2024rationalpointsellipsoidsmodular} or the study of particular light-cone Eisenstein series \cite{kelmer2023fourierexpansionlightconeeisenstein}, each offering unique insights and allowing for different generalisations, but yielding the same main result
\[N(T) \sim \frac{\omega_n}{n} T^n, \quad T\to\infty,\]
where $N(T)$ is the investigated quantity and $\omega_n$ is a coefficient depending on $n$. One striking feature of the approach based on light-cone Eisenstein series is that it yields an explicit, geometric interpretation of the leading constant $\omega_n$. 

Fix the quadratic form $Q_n(\bm v):=\sum_{j=1}^{n+1}v_j^2-v_{n+2}^2$ and its positive light-cone $\mathcal{V}_{Q_n}^+$. The unit sphere $S^n$ is obtained by projectivisation; if $\bm v\in \mathcal{V}_{Q_n}^+$, then $(v_1/v_{n+2},\dots,v_{n+1}/v_{n+2})\in S^n$. Thus rational points on $S^n$ correspond to primitive integer points on the positive light-cone. The group $G:=\operatorname{SO}_{Q_n}^+(\mathbb R)$ is the connected component preserving each sheet of the cone, and taking the lattice $\Gamma:=\operatorname{SO}_{Q_n}^+(\mathbb Z)$, the cusps of $\Gamma\backslash G$ correspond to $\Gamma$-orbits of rational rays in $\mathcal{V}^+_{Q_n}$, or equivalently to $\Gamma$-orbits of primitive integral points on the positive light-cone. Kelmer and Yu studied the associated light-cone Eisenstein series and showed that the residue at the first pole is expressed as a normalized sum of cusp volumes; with the Haar measure normalizations recalled below, their formula takes the form of \[\mu_G(\Gamma \backslash G)\omega_n=\sum_{i=1}^\kappa v_{P_i},\] where the $P_i$ are representatives of the cuspidal parabolic subgroups and $v_{P_i}$ are the corresponding cusp volumes (see below for the exact definitions). Based on explicit computations in small dimensions, they conjectured that this quantity is always rational.

\begin{center}
\scriptsize
\newcommand{\tall}{\rule[-10pt]{0pt}{25pt}}
\begin{tabular}{|c|c|c|c|}
\hline
$n$ & $\mu_G(\Gamma \backslash G)$ & $\omega_{Q_n}$ & $\mu_G(\Gamma \backslash G)\omega_{Q_n}$ \\
\hline
\tall $1$  & $\dfrac{\pi}{2}$ & $\dfrac{4}{\pi}$ & $2$ \\
\hline
\tall $2$  & $\dfrac{L(2, \chi_{-4})}{6}$ & $\dfrac{3}{L(2, \chi_{-4})}$ & $\dfrac{1}{2}$ \\
\hline
\tall $3$  & $\dfrac{\pi^2}{720}$ & $\dfrac{60}{\pi^2}$ & $\dfrac{1}{12}$ \\
\hline
\tall $4$  & $\dfrac{7\zeta(3)}{7680}$ & $\dfrac{80}{7\zeta(3)}$ & $\dfrac{1}{96}$ \\
\hline
\tall $5$  & $\dfrac{\pi^3}{388800}$ & $\dfrac{405}{\pi^3}$ & $\dfrac{1}{960}$ \\
\hline
\tall $6$  & $\dfrac{L(4, \chi_{-4})}{181440}$ & $\dfrac{63}{4L(4, \chi_{-4})}$ & $\dfrac{1}{11520}$ \\
\hline
\tall $7$  & $\dfrac{\pi^4}{457228800}$ & $\dfrac{2835}{17\pi^4}$ & $\dfrac{1}{161280}$ \\
\hline
\tall $8$  & $\dfrac{527\zeta(5)}{22295347200}$ & $\dfrac{512}{31\zeta(5)}$ & $\dfrac{17}{4354560}$ \\
\hline
\tall $9$  & $\dfrac{\pi^5}{164602368000}$ & $\dfrac{16065}{4\pi^5}$ & $\dfrac{17}{696729600}$ \\
\hline
\tall $10$ & $\dfrac{L(6, \chi_{-4})}{5748019200}$ & $\dfrac{165}{16L(6, \chi_{-4})}$ & $\dfrac{1}{557383680}$ \\
\hline
\tall $11$ & $\dfrac{691\pi^6}{3107034298368000}$ & $\dfrac{20945925}{2764\pi^6}$ & $\dfrac{31}{18393614400}$ \\
\hline
\tall $12$ & $\dfrac{8775\zeta(7)}{2448564210892800}$ & $\dfrac{512}{8775\zeta(7)}$ & $\dfrac{31}{1471492195200}$ \\
\hline
\end{tabular}
\end{center}

They conjectured that this holds for all $n$. This paper presents a proof of this statement by showing the stronger statement that every individual cusp volume $v_{P_i}$ is rational. In particular, the observed rationality is not a consequence of cancellation between cusps, but follows from the arithmetic structure of the unipotent stabilizers.

\begin{theorem} \label{conjecture}
    For every cusp $P_i$, the volume $v_{P_i}$ is rational.
\end{theorem}

\section{Preliminaries} \label{section 2}

We begin by introducing the necessary setting. In the following, we focus on conciseness, rather than details. Unless stated otherwise, for the proofs and detailed descriptions, see \cite{kelmer2023fourierexpansionlightconeeisenstein}. Also, note that all the presented results, including Theorem \ref{conjecture}, can easily be generalised to all rational quadratic forms of signature $(n+1,1)$, but for clarity purposes, we restrict ourselves to the case of the sphere.

As usual, we denote the $n+1$-dimensional upper half space as $\mathbb{H}^{n+1} := \{(\bm{x}, y) \in \mathbb{R}^n \times \mathbb{R}^+\}$, for $n\in \mathbb{N}.$ The hyperbolic space associated with $\Gamma$ may be then identified with $\mathbb H^{n+1}\cong G/K,$ where $K$ is a maximal compact subgroup of $G$.
   We define the cones
\[
\mathcal V_{Q_n}:=\{\bm v\in\mathbb R^{n+2}\setminus\{0\}:Q_n(\bm v)=0\},
\qquad
\mathcal V_{Q_n}^+:=\{\bm v\in \mathcal V_{Q_n}:v_{n+2}>0\},
\]
and fix $\bm e_0:=(-1,0,\dots,0,1)\in\mathcal V_{Q_n}^+.$ Let $P$ be the parabolic group stabilising the line spanned by $\bm e_0$. Then $P$ has Langlands decomposition $P=UAM,$ where $U\cong\mathbb R^n$, $A$ is one-dimensional, and $M\cong\operatorname{SO}_n(\mathbb R)$. We write $L:=UM$ for the stabiliser of $\bm e_0$. By the Iwasawa decomposition $G=UAK$, each $g\in G$ can be written as
\[
g=u_xa_yk,
\qquad x\in\mathbb R^n,\ y>0,\ k\in K.
\]

\begin{remark} \label{parametrisation}
    The matrix $u_x\in U$ can be explicitly parametrised by $x\in \mathbb{R}^n$.
    \[
u_x =
\begin{pmatrix}
1-\dfrac{\lVert x\rVert^2}{2} & x & \dfrac{\lVert x\rVert^2}{2} \\[6pt]
-x^t & I_n & x^t \\[6pt]
-\dfrac{\lVert x\rVert^2}{2} & x & 1+\dfrac{\lVert x\rVert^2}{2}
\end{pmatrix}
\]
    In particular, the parametrisation contains linear factors.
\end{remark}

With this notation, we normalise the Haar measures by
\[
\mu_K(K)=1, \quad d\mu_G(g)=y^{-(n+1)}\,d\bm x\,dy\,d\mu_K(k), \quad d\mu_L(u_xm)=d\bm x\,d\mu_M(m),
\]
where $\mu_M$ is the probability Haar measure on $M$.

For a parabolic subgroup $P_i$ of $G$, we write $\Gamma_{P_i}:=\Gamma\cap P_i$ and $\Gamma_{U_i}=\Gamma\cap U_i$. We define the \emph{cusps} of $\Gamma$ as the $\Gamma$-conjugacy classes of parabolic subgroups of $G$ such that $\Gamma_{U_i}$ is non-trivial. Moreover, we say $\Gamma$ has a \emph{cusp} at $[\bm v]$ if $\Gamma$ has a cusp at $P_{[\bm v]}$, being the unique parabolic subgroup fixing $[\bm v]$.

\begin{lemma} \label{rational}
    $\Gamma$ has a cusp at a ray $[\bm v]\subset\mathcal V_{Q_n}^+$ if and only if $[\bm v]$ is rational, i.e. has a rational representative.
\end{lemma}
\begin{proof}
    For proof, see  Proposition 3.15(ii) in \cite{fishman2021intrinsicdiophantineapproximationquadric}.
\end{proof}

It follows that the cusps correspond, after clearing denominators, to $\Gamma$-orbits of primitive integral points on the light cone. More precisely, there exist representatives $\bm v_1,\dots,\bm v_\kappa\in \mathcal V_{Q_n}^+(\mathbb Z)_{\operatorname{pr}}$ such that
\[
\mathcal V_{Q_n}^+(\mathbb Z)_{\operatorname{pr}}
=
\bigsqcup_{i=1}^{\kappa}\bm v_i\Gamma,
\]
and the rays $[\bm v_1],\dots,[\bm v_\kappa]$ form a full set of $\Gamma$-inequivalent cusps. 

We will also need the following result on the structure of the cusps.

\begin{lemma} \label{lemma:2.4}
    Let $\Gamma$ be a non-uniform lattice of $G$, and suppose $\Gamma$ has a cusp at $P'$ for some parabolic subgroup $P'=U'A'M'=U'M'A'$. Then $\Gamma_{P'} \le U'M'$ and $[\Gamma_{P'}:\Gamma_{U'}]<\infty$.
\end{lemma}

\begin{proof}
    See Lemma 2.1 in \cite{kelmer2023fourierexpansionlightconeeisenstein}.
\end{proof}

Now, for each cusp $P_i$ and a representative $\bm v_i$, choose an element $g_i\in G$ such that $\bm v_i=\bm e_0g_i^{-1}.$ We then define the corresponding cusp volume as $v_{P_i}:=\mu_L(g_i^{-1}\Gamma_{P_i}g_i\backslash L).$ 

\begin{lemma} \label{definition}
    The cusp volume is well-defined. In particular, it is independent of the choice of representative and $g_i^{-1}\Gamma_{P_i}g_i < L$, for $g_i$ as above. Moreover, $g_i^{-1}P_ig_i=P$.
\end{lemma}

\begin{proof}
    See Lemma 2.3 in \cite{kelmer2023fourierexpansionlightconeeisenstein}.
\end{proof}

Since conjugation transports decompositions, the above lemma also automatically implies that the group $P_i$ has (analogously to $P$) a Langlands decomposition $P_i=U_iA_iM_i$ with $U_i=g_iUg_i^{-1}, A_i=g_iAg_i^{-1}$ and $M_i=g_iMg_i^{-1}$.

We proceed with the definition of the normalised light-cone Eisenstein series. For $\Re(s)>n$, let
\[
E_{Q_n}(s,g):=\|\bm e_0\|^s
\sum_{\bm v\in\mathcal V_{Q_n}^+(\mathbb Z)_{\operatorname{pr}}}
\|\bm v g\|^{-s}.
\]

\begin{lemma} \label{eisenstein}
    The series $E_{Q_n}(s,g)$ has a meromorphic continuation to the whole $s$-plane, and is holomorphic in the half space $\Re(s)\ge n/2$ except for a simple pole at $s=n$ with constant residue
    $$\omega_n:=\mu_G(\Gamma \backslash G)^{-1}\sum_{i=1}^\kappa v_{P_i},$$
    and possibly finitely many simple poles in the interval $(\frac n2, n)$.
\end{lemma}

\begin{proof}
    See Corollary 3.5 in \cite{kelmer2023fourierexpansionlightconeeisenstein}.
\end{proof}

We can now return to our original problem of counting rational points with bounded denominators on $n$-dimensional spheres. Namely, for any $T\ge 1$, we define the counting function by
    \[N(T):=\#\Big\{\bm v\in \mathcal{V}_{Q_n}^+(\mathbb{Z})_{\operatorname{pr}}:\frac{\|\bm v\|}{\|\bm e_0\|} \le T \Big\}.\]
Furthermore, as we have
    \[\bm v \in \mathcal{V}_{Q_n}^+(\mathbb{Z})_{\operatorname{pr}} \implies  Q_n(\bm v)=0 \iff \sum_{k=1}^{n+1}\Big(\frac{v_k}{v_{n+2}}\Big)^2=1,\]
we can identify the counted $\bm v$ with the reduced rational points on the unit sphere. Then, for $\bm v \in \mathcal{V}_{Q_n}^+(\mathbb{Z})_{\operatorname{pr}}$, we have
\[\frac{\|\bm v\|}{\|\bm e_0\|}\le T \iff v_{n+2}\Biggl(\sum_{k=1}^{n+1}\Big(\frac{v_k}{v_{n+2}}\Big)^2+1 \Biggl)\le 2T \iff v_{n+2}\le T\]
and so our function $N(T)$ is equal to
\[N(T)=\#\Big\{\frac{\bm p}{q}\in S^n:(\bm p, q)\in \mathbb{Z}^{n+2}_{\operatorname{pr}}, 1\le q\le T \Big\}.\]
 
Using that, we can rewrite the normalised light-cone Eisenstein series for the identity element in $G$, $E_{Q_n}(s):=E_{Q_n}(s,1_G)$, as a classical Dirichlet series. Letting $v_{n+2}=:q$, we get $\sum_{i=1}^{n+1}v_i^2=q^2$ for all $\bm{v} \in \mathcal{V}_{Q_n}^+(\mathbb{Z})_{\operatorname{pr}}$, and so $\|\bm v\|^2=2q^2$. With $\|\bm e_0\|^2=2$, we get
    \[E_{Q_n}(s)=\sum_{k=1}^\infty a_kk^{-s}\]
with $a_k:=\#\{\bm v\in \mathcal{V}_{Q_n}^+(\mathbb{Z})_{\operatorname{pr}}:v_{n+2}=k\}=\#\Big\{\frac{\bm p}{q}\in S^n:(\bm p, q)\in \mathbb{Z}^{n+2}_{\operatorname{pr}}, q=k\Big\}.$
As, by Lemma~\ref{eisenstein}, $E_{Q_n}(s)$ is holomorphic for $s\ge n$ except for a simple pole at $s=n$, we apply the Wiener-Ikehara Theorem \cite{murty2024wiener} to obtain
    \[N(T)=\sum_{k=1}^Ta_k \sim \frac{\omega_{n}}{n}T^n,\]
where $\omega_n=\frac{1}{\mu_G(\Gamma \backslash G)}\sum_{i=1}^\kappa v_{P_i}$ is the residue at $s=n$.

\section{Rationality of $\mu_G(\Gamma \backslash G)\omega_{n}$}

To prove Theorem~\ref{conjecture}, we first need a simple lemma on the choice of $g$.

\begin{lemma}
Let $Q_n$ be a non-degenerate quadratic form on $\mathbb Q^{n+2}$, as before. Let $u,w\in \mathbb Q^{n+2}$
be non-zero isotropic vectors. If $u$ and $w$ lie on the same sheet of $V_{Q_n}^+(\mathbb R)$, then there exists $g\in \operatorname{SO}_{Q_n}^+(\mathbb Q)$
such that $ug=w$.
\end{lemma}

\begin{proof}
Since $Q_n$ is non-degenerate, we can choose $u',w'\in \mathbb Q^{n+2}$ such that $B_{Q_n}(u,u')=1,\; B_{Q_n}(w,w')=1.$ Replacing $u'\mapsto u'-Q_n(u')u, w'\mapsto w'-Q_n(w')w,$ we may also assume $Q_n(u')=Q_n(w')=0.$ Thus $(u,u')$ and $(w,w')$ are hyperbolic pairs over $\mathbb Q$.
\par Now, define an isometry between the two hyperbolic planes by $u\mapsto w,\;u'\mapsto w'$. By Witt's theorem over the field $\mathbb{Q}$ (as $\operatorname{char}\mathbb{Q}\ne2)$, this extends to some isometry $g_0\in O_{Q_n}(\mathbb Q)$ of the whole quadratic space with $ug_0=w.$
\par If $\det g_0=-1$, choose $v\in u^\perp\cap \mathbb Q^{n+2}$ with $Q_n(v)\neq 0$, and let
$$r_v(x)=x-\frac{2B_{Q_n}(x,v)}{Q_n(v)}\,v$$
be the corresponding reflection. Then $r_v\in O_{Q_n}(\mathbb Q)$, $\det r_v=-1$, and
$r_v(u)=u$. Hence $g:=r_v g_0\in SO_{Q_n}(\mathbb Q)$ and still satisfies $ug=w$. Since $u$ and $w$ lie on the same sheet of $V_{Q_n}^+(\mathbb R)$, this $g$ preserves that sheet, hence $g\in \operatorname{SO}_{Q_n}^+(\mathbb Q).$
\end{proof}

Now, we can proceed with the proof of the conjecture. The key point is that, after choosing rational cusp representatives, the corresponding conjugated unipotent stabilizers are lattices in $U\cong \mathbb R^n$ whose coordinate lattices are contained in $\mathbb Q^n$.

\begin{proof}[Proof of Theorem \ref{conjecture}]
    Per Lemma \ref{rational}, every ray with a cusp has at least one rational representative $\bm v_i\in \mathcal{V}_{Q_n}^+(\mathbb{Z})_{\operatorname{pr}}$. For all cusps, choose this $\bm v_i$, and then by the above lemma we can choose $g_i$ to be a matrix with rational entries satisfying $e_0g^{-1}_i=\bm v_i$.
    \par As $v_{P_i}=\mu_L(g_i^{-1}\Gamma_{P_i}g_i \backslash L)$, set $\Gamma_i:=g_i^{-1}\Gamma_{P_i}g_i \le L$ (as per Lemma \ref{definition}) and $\Gamma_{i_U}:=\Gamma_i\cap U$. By Lemma \ref{definition}, we have $g_i^{-1 }P_ig_i=P$ and $g_i^{-1}U_ig_i=U.$ Moreover, by Lemma \ref{lemma:2.4}, $[\Gamma_{P_i}:\Gamma_{U_i}]<\infty.$ Conjugating by $g_i$, we obtain $[g_i^{-1}\Gamma_{P_i}g_i:g_i^{-1}\Gamma_{U_i}g_i]=[\Gamma_i:\Gamma_{i_U}]<\infty$, as $U_i<P_i$.
    \par Take a fundamental domain $D_U\subset U$ for $\Gamma_{i_U} \backslash U$, then $D:=D_UM$ becomes a fundamental domain for $\Gamma_{i_U} \backslash L$. Its measure is
    $$\mu_L(D)=\int_M \int_{D_U} du \, d\mu_M(m)=\operatorname{vol}(D_U)\mu_M(M)=\operatorname{vol}(D_U).$$
    Therefore,
    $$\mu_L(\Gamma_i \backslash L)=\frac{\mu_L(D)}{[\Gamma_i : \Gamma_{i_U}]}=\frac{\operatorname{vol}(D_U)}{[\Gamma_i : \Gamma_{i_U}]}=\frac{\operatorname{vol}(\Gamma_{i_U} \backslash U)}{[\Gamma_i : \Gamma_{i_U}]},$$
    and it remains to show that $\operatorname{vol}(\Gamma_{i_U} \backslash U)$ is rational. Define $\Lambda:=\{x\in\mathbb R^n:u_x\in\Gamma_{i_U}\}$. As it is a full-rank lattice in $\mathbb{R}^n \cong U$, we have $\operatorname{vol}(\Gamma_{i_U} \backslash U)=\operatorname{covol}(\Lambda)$. Since $g_i\in \operatorname{SO}_{Q_n}^+(\mathbb Q)\subset G$ and $\Gamma_{P_i}\subset \operatorname{SO}_{Q_n}^+(\mathbb Z)$, every element of $\Gamma_i=g_i^{-1}\Gamma_{P_i}g_i$ has rational entries. In particular, for every $u_x\in\Gamma_{i_U}$, the matrix $u_x$ has rational entries. By the explicit parametrisation $x\mapsto u_x$ (cf. Remark~\ref{parametrisation}), the coordinates of $x$ can be read off from the linear entries of the matrix $u_x$, and we conclude $x\in\mathbb Q^n$ and $\Lambda\subset\mathbb Q^n$. Therefore $\Lambda$ admits a basis consisting of vectors in $\mathbb Q^n$, and so $\operatorname{covol}(\Lambda)\in\mathbb Q$. The claim follows.
\end{proof}

Using the formula $\omega_n=\mu_G(\Gamma \backslash G)^{-1}\sum_{i=1}^\kappa v_{P_i}$, we conclude the following

\begin{corollary}
    For $n\in \mathbb{N},$ we have $\mu_G(\Gamma \backslash G)\omega_n\in \mathbb{Q}$.
\end{corollary}

\newpage
\section*{Acknowledgements}
This note is based on a part of the author's bachelor's thesis, completed under the supervision of Violetta Weger and advised by Claire Burrin, to whom the author is particularly grateful for her guidance and helpful comments, as well as the proofreading of this paper.

\bibliographystyle{plain}          
\bibliography{references}   

\end{document}